\documentclass[12pt]{amsart}
\usepackage{graphicx}
\usepackage{amsmath, enumerate}
\usepackage{amsfonts, amssymb, amsthm, xcolor, stmaryrd}
\usepackage{tikz, tikz-cd, caption, tabu, longtable, mathrsfs, xtab, centernot, mathtools}
 \usepackage{dsfont, cite, todonotes}
\usepackage[pagebackref=true]{hyperref}
\usepackage{booktabs,rotating}
\numberwithin{equation}{section}
\theoremstyle{plain}

\newtheorem{thm}{Theorem}[section]
\newtheorem{lemma}[thm]{Lemma}
\newtheorem{prop}[thm]{Proposition}

\theoremstyle{definition}

\newtheorem{conj}{Conjecture}
\newtheorem{exmp}[thm]{Example}
\newtheorem{defi}{Definition}

\theoremstyle{remark}
\newtheorem{rmk}[thm]{Remark}

\newcommand{\Hb}{\mathbb{H}}
\newcommand{\SL}{{\mathrm{SL}}}
\newcommand{\Zb}{\mathbb{Z}}
\newcommand{\Cb}{\mathbb{C}}
\newcommand{\Qb}{\mathbb{Q}}
\newcommand{\Nb}{\mathbb{N}}
\newcommand{\lp}{\left (}
\newcommand{\rp}{\right )}
\newcommand{\pf}{\mathfrak{p}}

\newcommand{\Rb}{\mathbb{R}}
\newcommand{\smat}[4]{\left(\begin{smallmatrix}
                 #1 & #2\\
                 #3 & #4
\end{smallmatrix}\right)}

\newcommand{\GL}{{\mathrm{GL}}}

\newcommand{\Ac}{{\mathcal{A}}}

\newcommand{\ebf}{{\mathbf{e}}}

\newcommand{\SO}{{\mathrm{SO}}}
\newcommand{\Aut}{{\mathrm{Aut}}}

\newcommand{\Oc}{\mathcal{O}}
\newcommand{\df}{\mathfrak{d}}
\newcommand{\Cl}{{\mathrm{Cl}}}
\newcommand{\af}{\mathfrak{a}}
\newcommand{\Nm}{{\mathrm{Nm}}}
\newcommand{\ord}{{\mathrm{ord}}}

\renewcommand{\Mc}{{\mathcal{M}}}

\newcommand{\Uc}{{\mathcal{U}}}

\newcommand{\tr}{\mathrm{tr}}

\newcommand{\Ab}{\mathbb{A}}

\newcommand{\Pb}{\mathbb{P}}

\newcommand{\Qc}{\mathcal{Q}}
\newcommand{\OcF}{{\mathcal{O}_F}}
\newcommand{\dfF}{\mathfrak{d}_F}
\newcommand{\ga}{\gamma}
\newcommand{\ee}{A}
\AtBeginDocument{%
   \def\MR#1{}
}
\begin{document}
\title{ Span of Restriction of Hilbert Theta Functions}
  \author[Gabriele Bogo, Yingkun Li]{Gabriele Bogo, Yingkun Li}
\address{Fachbereich Mathematik,
Technische Universit\"at Darmstadt, Schlossgartenstrasse 7, D--64289
Darmstadt, Germany}
\email{bogo@mathematik.tu-darmstadt.de}
\email{li@mathematik.tu-darmstadt.de}
\thanks{}

\date{\today}
\maketitle

\begin{abstract}
  In this paper, we study the diagonal restrictions of certain Hilbert theta series for a totally real field $F$, and prove that they span the corresponding space of elliptic modular forms when the $F$ is quadratic or cubic.
  Furthermore, we give evidence of this phenomenon when $F$ is quartic, quintic and sextic. 
\end{abstract}

\section{Introduction}

Theta functions are classical examples of holomorphic modular forms.
Given a positive definite, unimodular $\Zb$-lattice $L$ of rank $8m$ with $m \in \Nb$, the associated theta function
\begin{equation}
  \label{eq:thetaL}
  \theta_L(\tau):= \sum_{\lambda \in L} q^{Q(\lambda)},~ q := \ebf(\tau) := e^{2\pi i \tau},
\end{equation}
is in $M_{4m}$, the space of elliptic modular forms of weight $4m$ on $\SL_2(\Zb)$. 
For example, the theta functions associated to the $E_8$ lattice and Leech lattice $\Lambda_{24}$ are explicitly given as
\begin{equation}
  \label{eq:E8}
  \theta_{E_8}(\tau) = E_4(\tau),~
  \theta_{\Lambda_{24}}(\tau) = E_4(\tau)^3 - 720 \Delta(\tau),
\end{equation}
where $E_{2k}(\tau)$ is the Eisenstein series of weight $2k$ and $\Delta(\tau)$ is the Ramanujan $\Delta$-function.

For $N \in \Nb$, we denote 
\begin{equation}
  \label{eq:McQ}
\Mc^{(N)}_\Qb := \bigoplus_{k \in \Nb}M_{Nk}
\end{equation}
the  finitely generated graded algebra of elliptic modular forms with weights divisible by $N$, and would like to consider the subalgebra~$    \Mc^\theta_\Qb \subset \Mc^{(4)}_\Qb$  generated by theta functions of unimodular lattices.
Using the relation
\begin{equation}
  \label{eq:thetaprod}
  \theta_{L_1 \oplus L_2} (\tau) =   \theta_{L_1}(\tau) \theta_{L_2} (\tau).
\end{equation}
for any two unimodular lattices $L_1, L_2$, we see that $\Mc^\theta_\Qb$ is simply the span of such theta functions.
Equation \eqref{eq:E8} and the fact $\Mc^{(4)}_\Qb = \Cb[E_4, \Delta]$ imply that 
\begin{equation}
  \label{eq:d0}
\Mc^\theta_\Qb = \Mc^{(4)}_\Qb.
\end{equation}

The construction of theta functions also extends to the case of Hilbert modular forms.
Let $F$ be  a totally real field of degree $d$ with ring of integers $\OcF$, and denote $\alpha_j \in \Rb$ the real embeddings of $\alpha \in F$ for $1 \le j \le d$.
For $N \in \Nb$, denote $\Mc_F^{(N)}$ the algebra of holomorphic Hilbert modular forms of parallel weight $Nk$ for $k \in \Nb$. 
Given a totally positive definite, $\Zb$-unimodular $\OcF$-lattice $L$ of rank $8m$ (see Definition \ref{def:Zuni}), the associated theta function
\begin{equation}
  \label{eq:thetaL1}
  \theta_L(\tau):= \sum_{\lambda \in L} \prod_{j = 1}^d q_j^{Q(\lambda)_j},~ \tau = (\tau_1, \dots, \tau_d) \in \Hb^d,~ q_j := \ebf(\tau_j), 
\end{equation}
is a Hilbert modular form of parallel weight $4m$ on $\SL_2(\OcF)$. 
It is well-known that such lattice exists precisely when
\begin{equation}
  \label{eq:mcond}
 m \in \frac{1}{d_2} \Nb,~  d_2 := \gcd(2, d)
\end{equation}
 (see Prop.\ \ref{prop:exist}). 
However, their explicit constructions and classification have only been carried out when $d$ is small (see e.g.\ \cite{Scharlau94, Wang14}). 
As a result, the relationship between $\Mc_F^{(4/d_2)}$ and the subalgebra $\Mc^\theta_F$ generated by such $\theta_L$ is not clear.

On the other hand, we have the following diagonal restriction map
\begin{align*}
  \Mc^{(N)}_F &\to   \Mc^{(Nd)}_\Qb \\
  f &\mapsto f^\Delta(\tau) := f(\tau^\Delta),
\end{align*}
where $\tau^\Delta = (\tau, \dots, \tau) \in \Hb^d$.
In this note, we will investigate the question about the image of $\Mc^\theta_F$ under this map, which is denoted by $(\Mc^\theta_F)^\Delta$ and contained in $\Mc^{(4d/d_2)}_\Qb$.
The main result is as follows.

\begin{thm}
  \label{thm:main}
  For a totally real field $F$ of degree $d = 2 , 3$, we have
  \begin{equation}
    \label{eq:main}
(    \Mc_F^\theta)^\Delta = \Mc^{(4d/d_2)}_\Qb.
\end{equation}
\end{thm}

Based on this, it is then natural to make the following conjecture.
\begin{conj}
  \label{conj:d}
Equation \eqref{eq:main} holds for any totally real field $F$ of degree $d$.
  \end{conj}

To prove Theorem \ref{thm:main}, we apply an instance of the Siegel-Weil formula to see that the Hecke Eisenstein series $E_{F, k}$ defined in \eqref{eq:EFk} is contained in $\Mc^\theta_F$ for all $k \in (4/d_2)\Nb$. 
Then we calculate the Petersson inner product between the diagonal restriction of $E_{F, k}$ and an elliptic cusp form. 
For $d = 2$, this inner product is related to Fourier coefficients of half-integral weight modular forms by a result of Kohnen-Zagier \cite{KZ84}. 
For $d \ge 3$, we give an expression for this inner product in terms of a sum over the double coset $\Gamma_{F, \infty} \backslash \Gamma_F /\Gamma_\Qb$ (see Prop.\ \ref{prop:Pet}). 
When $d = 3$, we related this double coset to orders in a cubic field $F$ (see section \ref{sec:3}).
Using these results, we show that when $d = 2, 3$, $\Mc_\Qb^{(4d/d_2)}$ can be generated by $E_{F, k}^\Delta$ and $\theta_L^\Delta$ for a $\Zb$-unimodular $\OcF$-lattice $L$.

The same approach can be used to check conjecture \ref{conj:d} numerically when $d \in \{4, 5, 6, 8, 10\}$.
We list some results for $d = 4, 5, 6$ and $F$ has small discriminants in the last section (see Theorem~\ref{thm:num}).


\textbf{Acknowledgement}:
We thank Jan Bruinier and Tonghai Yang for helpful discussion about Prop.\ \ref{prop:Funilatt}. 
The authors are supported by the LOEWE research unit USAG, and by the Deutsche Forschungsgemeinschaft (DFG) through the Collaborative Research Centre TRR 326 ``Geometry and Arithmetic of Uniformized Structures'', project number 444845124.

\section{Preliminary}
Let $F$ be a totally real field of degree $d$ with ring of integers $\OcF$ and different $\dfF$.
Denote $\Cl(F)$ the (wide) class group of $F$.
Let $(V, Q)$ be an $F$-quadratic space of dimension $n$.
We say that $V$ is \textit{totally positive} if $V \otimes_{\iota(F)} \Rb$ is totally positive for every real embedding $\iota: F \hookrightarrow \Rb$.
In that case, $\SO_V(\Rb)$ is compact and the double quotient $\SO_V(F)\backslash \SO_V(\hat F)/K$ is a finite set for any open compact subgroup $K \subset \SO_V(\hat F)$. 
Here $\Ab_F$ and $\hat F$ are the adele and finite adele of $F$.

A finitely generated $\OcF$-module $L \subset V$ is called a ($\OcF$-)lattice if $L \otimes_\OcF F = V$.
We denote $\hat L := L \otimes \hat \Zb \subset \hat V = V \otimes \hat \Qb$.
If $Q(L) \subset \dfF^{-1}$, we say that $L$ is \textit{$\Zb$-even integral}
and call the lattice
\begin{equation}
  \label{eq:L'}
  L' := \{y \in V: (y, L) \subset \dfF^{-1} \}
\end{equation}
its \textit{$\Zb$-dual}.
Viewed as a $\Zb$-lattice  with respect to $Q_\Qb(x):= \tr_{F/\Qb}Q(x)$, such $L$ is even integral with dual $L'$. 
\begin{defi}
  \label{def:Zuni}
   An $\OcF$-lattice $L$ is said to be \textit{$\Zb$-unimodular} if $L' = L$.
\end{defi}
As a convention, the trivial lattice in the trivial $F$-vector space is totally positive and $\Zb$-unimodular.
Consider the monoid
\begin{equation}
  \label{eq:UF+}
  \Uc^+_F := \{(L, Q): L \text{ is an even } \Zb\text{-unimodular } \OcF\text{-lattice and totally positive definite}\}
\end{equation}
with respect to $\oplus$, and denote $\Uc^{+, n}_F \subset \Uc^+_F$ the subset of lattices of rank $n$.
We first have the following result.
\begin{prop}
  \label{prop:exist}
  The set $\Uc^{+, n}_F$ is non-empty precisely when $(8/ d_2) \mid n$. \end{prop}
\begin{proof}
Satz 1 in \cite{Chang70} implies that there exists definite, unimodular $\OcF$-lattices in the sense loc.\ cit. if and only if $(8/d_2)\mid n$.
Furthermore since $n$ is even, all of the $2^d$ possible definite signatures will appear in the set of definite, unimodular $\OcF$-lattices of rank $n$. 
One can then use the fact that the class $\dfF$ in the class group is a square to translate this result to the existence $\Zb$-unimodular lattices.
 (see the proof of Prop.\ 2.5 in \cite{Li21} for details).
\end{proof}
\begin{rmk}
  \label{rmk:hL}
  For $L \in \Uc^{+, n}_F$ and $h \in \SO_V(\hat\Qb)$ with $V  = L \otimes_\OcF F$, the lattice
  \begin{equation}
    \label{eq:hL}
h \cdot L :=   (h \cdot \hat L) \cap V \subset V  
\end{equation}
   is also in $\Uc^{+, n}_F$.
\end{rmk}

For each $L \in \Uc^{+, n}_F$, let $\theta_L(\tau)$ be the associated theta function defined in \eqref{eq:thetaL1}.
It is a Hilbert modular form of parallel weight $n/2$ for $\SL_2(\OcF)$. 
Now, the Siegel-Weil formula \cite{Siegel, Weil65} gives us the following result.
\begin{prop}
  \label{prop:Funilatt}
  Let $F$ be a totally real field of degree $d$.
Then
\begin{equation}
  \label{eq:SW}
  \int_{\SO_V(F)\backslash \SO_V(\Ab_F)/\SO_V(\Rb)}
  \theta_{h \cdot L}(\tau) dh = \kappa E_{F, n/2}(\tau),
\end{equation}
for some positive constant $\kappa$, where $E_{F, k}$ is the Hecke Eisenstein series of parallel weight $k$ defined by
\begin{equation}
  \label{eq:EFk}
  \begin{split}
      E_{F, k}(\tau) 
&:= 
1 + \zeta_F(k)^{-1} \sum_{\Ac = [\af] \in \Cl(F)} \Nm(\af)^k 
\sum_{(c, d) \in \af^2/\Oc_F^\times,~ c \neq 0} 
\prod_{j = 1}^{d} (c_j \tau_j + d_j)^{-k}
  \end{split}
\end{equation}
In particular, $E_{F, k} \in \Mc^\theta_F$ for all $k \in (4/d_2)\Nb$. 
\end{prop}

\begin{rmk}
  \label{rmk:FEHE}
The Hecke Eisenstein series have the following well-known Fourier expansion (see \cite{Siegel-rational, Zag76})
\begin{equation}
  \label{eq:FEHE}
E_{F,k}(\tau)=1 + \frac{2^d}{\zeta_F(1-k)}\sum_{t \in \df_F^{-1},~ t \gg 0}
\sigma_{k-1}(t\df_F) \prod_{j = 1}^d q_j^{t_j \tau_j}
\end{equation}
with $  \sigma_r(\af):=\sum_{{\mathfrak{b}|\mathfrak{a},~ \mathfrak{b}\subset\mathcal{O}_F}}{\Nm(\mathfrak{b})^r}$ for any integral ideal $\af$ and $r \in \Nb$.
\end{rmk}

\begin{proof}
By the Siegel-Weil formula, the left hand side of \eqref{eq:SW} equals to the Eisenstein series
$$
E_L(\tau) = v^{-n/4} \sum_{\gamma \in B(F) \backslash \SL_2(F)} \Phi_L(\gamma g_\tau, n/2 - 1),
$$
where $B \subset \SL_2$ is the standard Borel subgroup, and $\Phi_L$ is the Siegel-Weil section associated to the lattice $L$ (see e.g.\ \cite[section I.3]{Kudla08}).
For $t \in F^\times$, the $t$-th Fourier coefficient of $E_L$ is given by 
$$
\prod_{\pf < \infty} W_{t, \pf}(1, n/2-1, \Phi_{L, \pf})
$$
up to constant independent of $t$.
Here $W_{t, \pf}(g, s, \phi)$ is the local Whittaker function (see e.g.\ \cite{Yang05}). 
Since $L$ is $\Zb$-unimodular, the local lattice $L \otimes \Oc_{F, \pf}$ in $V \otimes F_\pf$ is self-dual for every finite place $\pf$.
Standard calculations (see e.g.\ \cite{KY10}) then gives us
$$
W_{t, \pf}(1, s, \Phi_{L, \pf})
= \sum_{m = 0}^{\ord_\pf(t \df_{F_\pf})} \Nm(\pf)^{s}
$$
when $t \in \df_{F_\pf}^{-1}$, and zero otherwise.
So up to a constant, the Eisenstein series $E_L$ and $E_{F, n/2}$ have the same non-constant term Fourier coefficients, hence agree.
Now the left hand side of \eqref{eq:SW} is just a sum of $\theta_{L_j}$ over certain $L_j \in \Uc^{+, n}_F$ by Remark \ref{rmk:hL}.
Combining this with Prop.\ \ref{prop:exist} finishes the proof.
\end{proof}

We can rewrite the Hecke-Eisenstein series $E_{F, k}$ as
$$
      E_{F, k}(\tau) 
:= 
1 + 
 \sum_{\substack{\Ac = [\af] \in \Cl(F)\\(c, d) \in \af^2/\Oc_F^\times\\c \neq 0\\ \OcF c + \OcF d = \af}} 
\lp \frac{ \Nm(\af)}{\Nm(c)} \rp^k 
\prod_{j = 1}^{d} (\tau_j + d_j/c_j)^{-k}
$$
For any $\beta \in F$, there is unique $\Ac = [\af]$ and $(c, d) \in \af^2/\Oc_F^\times$ with $c \neq 0$ such that $\af = \OcF c + \OcF d$ and  $\beta = d/c$. Therefore, we denote
\begin{equation}
  \label{eq:ebeta}
\ee_{\beta} := 
 \frac{ \Nm(c)}{\Nm(\af)} \in \Zb - \{0\} .
\end{equation}
It is easy to check this definition does not depend on the choice of the representative $\af$, and
\begin{equation}
  \label{eq:eprop}
  \ee_{\beta + a, k} =   \ee_{\beta, k}
\end{equation}
for all $a \in \Zb$.
Then we have
\begin{equation}
  \label{eq:HE2}
        E_{F, k}(\tau) = 1 + \sum_{\beta \in F} \ee_{\beta}^{-k} \prod_{j = 1}^{d} (\tau_j + \beta_j)^{-k}.
\end{equation}

%

\section{Petersson Inner Product Calculations}
\label{sec:Pet}
In this section, let $F/\Qb$ be totally real with degree $d \ge 3$.
We will give an expression for the Petersson inner product between the diagonal restriction of the Hecke Eisenstein series $E_{F,k}$ and an elliptic cusp form $f$ of weight $dk$.

For $\alpha \in M_{m, n}(F)$ and   $1 \le j \le d$, we write $\alpha_j \in M_{m, n}(\Rb)$ with $1 \le j \le d$ for the real embeddings of $\alpha$.
We identify  $\Pb^1(F) \cong B(F) \backslash \SL_2(F)$ via
\begin{equation}
  \label{eq:iota}
\beta \mapsto
\begin{cases}
 \smat{*}{*}{1}{\beta}& \beta \in F,\\  
 \smat{*}{*}{0}{1}& \beta = \infty.
\end{cases}
\end{equation}
Let $S_0 \cup\{\infty\} \subset \Pb^1(F)$ be a set of representatives of the double coset $B(F)\backslash \SL_2(F) /\SL_2(\Zb)$.
Then $S_0 \subset  F - \Qb$ 
and we can use \eqref{eq:HE2} to express the diagonal restriction of $E_{F, k}$ as
\begin{equation}
\label{eq:ED}
  E_{F, k}^\Delta(\tau) 
= E_{dk} + \sum_{\beta \in S_0} E_{F, k, \beta}(\tau),~
 E_{F, k, \beta}(\tau) := \sum_{\gamma \in \SL_2(\Zb)} 
\ee_{-\gamma^{-1}\cdot (-\beta_j)}^{-k}
\prod_{j = 1}^{d} 
(\tau - \gamma^{-1}\cdot (-\beta_j))^{-k}
\end{equation}
with $\tau \in \Hb$.
Note that $E_{d k}$ is just the elliptic Eisenstein series of weight $dk$.

Let $f(\tau) = \sum_{n \ge 1} c_n q^n \in S_{d k}$ be a cusp form. 
We are interested in estimating its inner product with $E^\Delta_{F, k}$. 
By the usual unfolding process, we obtain
\begin{align*}
  \langle E^\Delta_{F, k}, f\rangle
&= \sum_{\beta \in S_0} \int_{\Gamma_{\infty} \backslash \Hb}
E^\infty_{F, k, \beta}(\tau) \overline{f(\tau)} v^{d k} \frac{dudv}{v^2}\\
&= \sum_{\beta \in S_0} 
\int_0^\infty \sum_{n \ge 1}  \overline{c_n} a_{F, k, \beta}(n, v) e^{-2\pi n v} v^{d k - 1}\frac{dv}{v},
\end{align*}
where $\Gamma_\infty := B(\Qb) \cap \SL_2(\Zb)$ and 
\begin{equation}
  \label{eq:Einf}
  \begin{split}
  E^\infty_{F, k, \beta}(\tau) 
&:= \sum_{\gamma \in \Gamma_{ \infty}} 
\ee_{-\gamma^{-1}\cdot (-\beta)}^{-k}
\prod_{j = 1}^{d} (\tau - \gamma^{-1}\cdot (-\beta_j))^{-k}\\
&= 
2 \ee_{\beta}^{-k}
\prod_{j = 1}^{d} ( \tau + \beta_j + b)^{-k}
= \sum_{n \in \Zb} a_{F, k, \beta}(n, v) \ebf(nu).
\end{split}
\end{equation}
for $\beta = d/c \in S_0$.
Here we have $r_{-\gamma\cdot (-\beta)} = r_\beta$  for all $\gamma \in \Gamma_\infty$ by \eqref{eq:eprop}. 
It is easy to see that 
\begin{equation}
  \label{eq:an}
  \begin{split}
      a_{F, k, \beta}(n, v) 
      &= 
2 \ee_\beta^{-k}
      \int_{\Rb} \prod_{j = 1}^{d} ( u + iv + \beta_j)^{-k} \ebf(-nu) du\\
      &=  4\pi i
(-\ee_\beta)^{-k} 
      \sum_{z \in Z(\beta)} \mathrm{Res}_{x = z}
\lp \ebf(nx) \prod_{j = 1}^{d} ( x - (\beta_j + iv))^{-k}  \rp,
  \end{split}
\end{equation}
where $Z(\beta) := \{\beta_j + iv: 1 \le j \le d\} \subset \Hb$ since
\begin{equation}
  \label{eq:Pj}
 \sum_{z \in Z(\beta)}   \mathrm{Res}_{x = z} \lp \ebf(nx) \prod_{j = 1}^{d} ( x - z_j)^{-k} \rp
=   \frac{1}{2\pi i} \int_\Rb \ebf(nx)  \prod_{j = 1}^{d} ( x - z_j)^{-k} dx.
\end{equation}
Suppose $\beta_j$'s are all distinct.
Then
\begin{align*}
  \sum_{z \in Z(\beta)} \mathrm{Res}_{x = z}
  &\lp \ebf(nx) \prod_{j = 1}^{d} ( x - (\beta_j + iv))^{-k} \rp
  =
\frac{1}{\Gamma(k)}    \sum_{j = 1}^{d}
\lp \frac{d}{dx}\rp^{k-1}
    \lp 
    \frac{ \ebf(nx)}{\prod_{j' = 1,~ j' \neq j}^{d}( x - (\beta_{j'} + iv))^{k}} \rp\mid_{x = \beta_{j'} + iv}\\
  &=
\frac{\ebf(n iv)}{\Gamma(k)}    \sum_{j = 1}^{d}
    \sum_{\ell = 0}^{k-1} (2\pi i n )^{k-1-\ell}
    {\ebf(n\beta_{j})e^{-2\pi n v}}
    \binom{k-1}{\ell}
    \lp \frac{P_{d-1, k, \ell}}{Q_{d - 1, k+\ell}} \rp(\beta_{j} - \beta_1, \dots, \beta_{j} - \beta_{d}),
\end{align*}
where $P_{m, k, \ell}, Q_{m, r} \in \Qb[x_1, \dots, x_m]$ are symmetric polynomials of degrees $(m-1)\ell$ and $mr$ defined by
\begin{equation}
  \label{eq:PQ}
  \begin{split}
    P_{m, k, \ell}(x_1, \dots, x_m)
    &:=
    (x_1 \dots x_m)^{k+\ell}  (\partial_{x_1} + \dots + \partial_{x_m})^{\ell} (x_1 \dots x_m)^{-k},\\
    Q_{m, r}(x_1, \dots, x_m) &:= (x_1 \dots x_m)^{r}.
  \end{split}
\end{equation}
Note that
\begin{equation}
  \label{eq:PQ1}
  \frac{P_{m, k, \ell}}{Q_{m, k+\ell}}(x_1, \dots, x_m)
  =
(-1)^\ell {\ell!}  \sum_{r = (r_j) \in \Nb^m,~ \sum_{j} r_j = \ell}
\lp\!\! \binom{k}{r}\!\!\rp 
    \prod_{j = 1}^m x_j^{-k-r_j},
  \end{equation}
  where $\lp\!\! \binom{k}{r}\!\!\rp := \frac{k^{(r_1)}\dots k^{(r_m)}}{r_1!\dots r_m!}$
  for $r = (r_1, \dots, r_m) \in \Nb^m$  with $k^{(n)} := k(k+1)\dots (k+n-1)$.
  Substituting this into the unfolding gives us the following result.
  \begin{prop}
    \label{prop:Pet}
    Suppose $F$ is a totally real field of degree $d \ge 3$ and there is no intermediate field between $F$ and $\Qb$. For any $k \in 2\Nb$ and $f(\tau) = \sum_{n \ge 1} c(n)q^n \in S_{dk}$, we have
\begin{equation}
  \label{eq:inner0}
  \begin{split}
  &  \langle E^\Delta_{F, k}, f\rangle
  =
 \frac{ i \Gamma(d k - 1)}{(4\pi)^{d k - 2} \Gamma(k)}    
    \sum_{\ell = 0}^{k-1} (2\pi i)^{k-1-\ell}
    \sum_{\beta \in S_0}    \ee_{\beta}^{-k} \\ 
&\times
\sum_{j  = 1}^{d}      
\lp    \frac{P_{d-1, k, \ell}}{Q_{d - 1, k+\ell}}\rp
    (\beta_{j } - \beta_1, \dots, \beta_{j } - \beta_{j -1}, \beta_{j } - \beta_{j  + 1}, \dots, \beta_{j } - \beta_{d})
    \sum_{n \ge 1}
\frac{    \ebf(n\beta_{j })    \overline{c_n}}{    n^{(d -1)k + \ell }},
\end{split}
\end{equation}
where the polynomials $P_{m, k, \ell}$ and $Q_{m, r}$ are defined in \eqref{eq:PQ}. 
  \end{prop}

  \begin{rmk}
    The condition that there is no intermediate field between $F$ and $\Qb$ implies that $\beta_i = \beta_j$ if and only if $i = j$ for all $\beta \in F - \Qb$.
A similar but more complicated formula for the inner product can be derived without this condition.
  \end{rmk}
  \begin{exmp}
    \label{exmp:d=3}
  Let $d = 3$ and $k = 2$. Then 
$$
\frac{P_{d-1, k, \ell}}{Q_{d - 1, k+\ell}}(x, y)
=
\begin{cases}
1/(xy)^2,  & \ell = 0,\\
-2(x+y)/(xy)^3,&\ell = 1.
\end{cases}
$$
Set $\ga_1 := \beta_2 - \beta_3, \ga_2 := \beta_3 - \beta_1, \ga_3 := \beta_1 - \beta_2$,
we have 
  \begin{align*}
 &         \sum_{\ell = 0}^{k-1} (2\pi i n)^{k-1-\ell}
\sum_{j  = 1}^{d}      
\frac{P_{d-1, k, \ell}}{Q_{d - 1, k+\ell}}(\beta_{j } - \beta_1, \dots, \beta_{j } - \beta_{d})
    \ebf(n\beta_{j })\\
&= 
\lp \frac{2\pi i n}{(\gamma_2\gamma_3)^2}  + \frac{2(\gamma_3 - \gamma_2)}{(\gamma_2\gamma_3)^3}\rp    \ebf(n\beta_{1})
+ \lp \frac{2\pi i n}{(\gamma_1\gamma_3)^2}  + \frac{2(\gamma_1 - \gamma_3)}{(\gamma_1\gamma_3)^3}\rp    \ebf(n\beta_{2})
+ \lp \frac{2\pi i n}{(\gamma_1\gamma_2)^2}  + \frac{2(\gamma_2 - \gamma_1)}{(\gamma_1\gamma_2)^3}\rp    \ebf(n\beta_{3}).
  \end{align*}
  For $d = 3$ and $k-1 \ge \ell \ge 0$, we can write explicitly
  \begin{align*}
    &
      \sum_{j  = 1}^{d}      
\frac{P_{d-1, k, \ell}}{Q_{d - 1, k+\ell}}(\beta_{j } - \beta_1, \dots, \beta_{j } - \beta_{d}) 
    \ebf(n\beta_{j })\\
&=       
      \frac{P_{2, k, \ell}}{Q_{2, k+\ell}}(\ga_3, -\ga_2)\ebf(n\beta_{1})+
      \frac{P_{2, k, \ell}}{Q_{2, k+\ell}}(-\ga_3, \ga_1)\ebf(n\beta_{2})+
      \frac{P_{2, k, \ell}}{Q_{2, k+\ell}}(-\ga_2, -\ga_1)\ebf(n\beta_{3})      .
  \end{align*}
  Using the inequalities $k^{(a)}k^{(b)} \le k^{(a+b)}$, $(x_1 + x_2 + x_3)^2 \le 3(x_1^2 + x_2^2 + x_3^2)$, 
  \begin{equation}
    \label{eq:ineq1}
\sum_{\sigma \in S_3} x_{\sigma(1)}^a x_{\sigma(2)}^b x_{\sigma(3)}^c \le \frac{a!b!c!}{(a+b+c)!}(x_1 + x_2 + x_3)^{a+b+c},~ x_i, a, b, c \ge 0    
  \end{equation}
and  Equation \eqref{eq:PQ1}, we obtain the bound
  \begin{align*}
    &
\left|      \sum_{j  = 1}^{d}      
\frac{P_{d-1, k, \ell}}{Q_{d - 1, k+\ell}}(\beta_{j } - \beta_1, \dots, \beta_{j } - \beta_{d}) 
      \ebf(n\beta_{j })\right|\\
& \le     
\left|      \frac{P_{2, k, \ell}}{Q_{2, k+\ell}}(\ga_3, -\ga_2) \right|+
\left|      \frac{P_{2, k, \ell}}{Q_{2, k+\ell}}(-\ga_3, \ga_1)\right|+
\left|      \frac{P_{2, k, \ell}}{Q_{2, k+\ell}}(-\ga_2, -\ga_1)\right|\\
    & \le \frac{\ell! }{|\ga_1\ga_2\ga_3|^{k+\ell}}
      \sum_{a + b = \ell} \frac{k^{(a)}k^{(b)}}{a! b!}
\lp      |\ga_1^b \ga_2^a \ga_3^{k+\ell}| +       |\ga_2^b \ga_3^a \ga_1^{k+\ell}| +       |\ga_3^b \ga_1^a \ga_2^{k+\ell}| \rp\\
    & \le \ell!  \frac{(|\ga_1| + |\ga_2| + |\ga_3|)^{k+2\ell}}{|\ga_1\ga_2\ga_3|^{k+\ell}}
      \frac{      (k+\ell)!}{(k + 2\ell)!}
      \frac{\ell + 1}{2} k^{(\ell)}
\le
      \frac{      \binom{k-1+\ell}{\ell}}{\binom{k+2\ell}{\ell}}
      (\ell + 1)!
\frac{3^{k/2 + \ell}}{2}      \frac{(\ga_1^2 + \ga_2^2 + \ga_3^2)^{k/2+\ell}}{|\ga_1\ga_2\ga_3|^{k+\ell}}.
  \end{align*}  
\end{exmp}

\section{Double Coset and Binary Cubic Forms}
\label{sec:3}
  When $d = 3$, we can identify the double coset $B(F)\backslash \SL_2(F) /\SL_2(\Zb) - \{\infty\}$ with orders in $\OcF$ in the following way.
  Let
  $$
  \Qc_F := \{f(X, Y) = A  X^3 + B  X^2 Y + C  XY^2 + D  Y^3 \in \Zb[X, Y]: f(\beta, 1) = 0 \text{ for some } \beta \in F \backslash \Qb\}
  $$
  be the set of integral binary cubic forms with a root in $F - \Qb$.
  A form is primitive if its coefficients have no common factor.
  There is a natural action of $\SL_2(\Zb)$ on $\Qc_F$ that preserves the discriminant
  \begin{equation}
    \label{eq:Df}
    \begin{split}
      \Delta(f)
      &:= A^6 ((\beta_1 - \beta_2)(\beta_1 - \beta_3)(\beta_2 - \beta_3))^2\\
      &= 18 ABCD + B^2C^2 - 4 A C^3 - 4 B^3 D - 27 A^2D^2,    
    \end{split}
  \end{equation}
 and the subset of primitive forms.
  The quantity
  \begin{equation}
    \label{eq:Pf}
  P(f) := B^2 - 3A C > 0    
  \end{equation}
is the leading coefficient of the Hessian of $f$, which is a positive definite quadratic form and a coinvariant of $f$. 
  For every $f \in \Qc_F$, Prop.\ 2 in \cite{Cremona99} gives us $f' \sim_{\SL_2(\Zb)} f$ satisfying
  \begin{equation}
    \label{eq:Pbd}
    P(f') \le \sqrt{\Delta(f')} = \sqrt{\Delta(f)}. 
  \end{equation}

  Given $\beta  \in F - \Qb$, 
%
  we can associate to it a primitive element $f_\beta \in \Qc_F$ defined by
  \begin{equation}
    \label{eq:fs}
    f_\beta(X, Y) := \ee_\beta^{} \prod_{j = 1}^3 (X - \beta_j Y) = A_\beta X^3 + B_\beta X^2 Y + C_\beta XY^2 + D_\beta Y^3 \in \Qc_F.
  \end{equation}
  Note that $f_\beta(\beta, 1) = 0$ and the right action of $\SL_2(\Zb)$ on $ B(F)\backslash \SL_2(F)$ corresponds to its natural action on $\Qc_F$.

  To any binary cubic form $f$ with non-zero discriminant and $f(\beta, 1) = 0$
  we can associate the free $\Zb$-module of rank 3
  \begin{equation}
    \label{eq:Of}
    \Oc_f := \Zb + \Zb A \beta + \Zb (A \beta^2 + B \beta + C) \subset \Qb(\beta),
  \end{equation}
  which is also a  commutative ring.
  A classical result of Delone and Faddeev tells us that this gives a bijection between $\GL_2(\Zb)$-classes of binary cubic forms with non-zero discriminants and isomorphism classes of commutative rings that are free $\Zb$-modules of rank 3  \cite{DF64}.
  If we restrict $\beta$ to be in a fixed field $F$, then $\Oc_f$ is an order in $\OcF$, and $\Oc_{f_1}, \Oc_{f_2} \subset \OcF$ are the same if and only if $f_1, f_2 \in \Qc_F$ are  $\GL_2(\Zb)$-equivalent (see e.g.\ \cite[Lemma 3.1]{Na98}).
  Furthermore, we have
  \begin{equation}
    \label{eq:disc}
    \Delta(f) = \Delta(\Oc_f) = D_F [\OcF:\Oc_f]^2
  \end{equation}
  with $\Delta(\cdot)$ the discriminant. 
  For $s = [\beta] \in \Pb^1(F) /\SL_2(\Zb) - \{\infty\}$, we then denote
  \begin{equation}
    \label{eq:Ocs}
  \Oc_s := \Oc_{f_\beta}, \Delta(s) := \Delta(\Oc_s).    
\end{equation}
The discussions above lead to the following result.
\begin{prop}
  \label{prop:DF}
  The map
  \begin{align*}
    \Pb^1(F) /\SL_2(\Zb) - \{\infty\}
    &\to \{\Oc: \Oc \subset \OcF \text{ is an order}\}/\cong\\
    s &\mapsto \Oc_s
  \end{align*}
  is well-defined and $(2 |\mathrm{Aut}(\OcF)|)$-to-1.
\end{prop}
\begin{rmk}
  \label{rmk:1}
  The quantity $|\mathrm{Aut}(\OcF)|$ is either 3 or 1 depending on $F/\Qb$ is Galois or not.
\end{rmk}

  Finally, the following Dirichlet series
    \begin{equation}
    \label{eq:etaF}
    \begin{split}
      \eta_F(s)
      &:= \sum_{\Oc \subset \OcF \text{ order} }[\OcF: \Oc]^{-s}
    =  \sum_{\Oc\subset \OcF \text{ order} }\frac{D_F^{s/2}}{\Delta(\Oc)^{s/2}}.
    \end{split}
  \end{equation}
can be factorized in the following way by a result of Datskovsky and Wright \cite{DW86} (see \cite[Lemma 3.2]{Na98}) 
  \begin{equation}
    \label{eq:DW}
    \eta_{F}(s) = \frac{\zeta_F(s)}{\zeta_F(2s)} \zeta(2s)\zeta(3s-1).
  \end{equation}

\section{Proof of Theorem \ref{thm:main}}
We are now ready to prove Theorem \ref{thm:main}. 
The cases of $d = 2, 3$ are proved separately.

\begin{proof}[Proof of Theorem \ref{thm:main} for $d = 2$]
  For $k = 2, 4$, the space $M_{2k}$ is 1-dimensional and spanned by the Eisenstein series $E_{2k}$. 
Since $\theta_L^\Delta$ is non-trivial for any $L \in \Uc^+_F$, the claim follows in these two base cases as $M^\theta_{F, k}$ is non-trivial by Prop.\ \ref{prop:exist} (see also \cite{Scharlau94} for an explicit construction).
More generally, we know that $\Mc_{\Qb}^{(4)} = \Qb[E_4, \Delta]$. Therefore, it suffices to show that $\Delta \in S_{12}$ is in $(M^\theta_{F, 6})^\Delta$.
As $M_{12}$ is 2-dimensional and
\begin{equation}
  \label{eq:E43}
E_{4}^3 = E_{12} + \frac{432000}{691} \Delta  \in  (M^\theta_{F, 6})^\Delta,  
\end{equation}
we just need to produce a form $f \in (M^\theta_{F, 6})^\Delta$ linearly independent from $E_4^3$. 
For this purpose, we apply Prop.\ \ref{prop:Funilatt} with $k = 6$ to get
$$
f(\tau):= (E_{F, 6}^\Delta)(\tau) = 1 + \frac{4}{\zeta_F(-5)} 
\sum_{m \ge 1} q^m
\sum_{\nu \in \dfF^{-1},~ \nu \gg 0,~ \tr(\nu) = m} \sigma_{5}((\nu)\dfF).
$$
By Theorem 6 in \cite{KZ84}, we know that 
\begin{equation}
  \label{eq:REis}
  f =  E_{12} - \frac{12}{691} \frac{c(D)}{\zeta_F(-5)} \Delta,
\end{equation}
where $c(D)$ is the $D$-th Fourier coefficient of the half-integral weight form
$$
g(\tau) = \sum_{D \in \Nb} c(D) q^D := \frac{1}{8\pi i}(2 E_4(4\tau) \theta'(\tau) - E_4'(4\tau) \theta(\tau))
$$ 
spanning the Kohnen plus space $S_{13/2}^+$. 
Now using the easy estimate $L(k, \chi_D) > 2 - \zeta(k)$ for $k \ge 2$ (see e.g.\ Equation (3) in \cite{ChoieKohnen13}) we know that $\zeta_F(1-k)  = D^{k-1/2} \frac{4\Gamma(k)^2}{(-4\pi)^k} \zeta_F(k) $ satisfies
$$
|\zeta_F(-5)| >
0.01\cdot 
D^{11/2}.
$$
On the other hand, the Hecke bound for $c(D)$ yields 
$$
|c(D)| \le  c \cdot D^{13/4},~ c := e^{2\pi } \max_{\tau \in \Hb} |g(\tau)|v^{13/4} < 10
$$ 
Comparing with \eqref{eq:E43}, it is clear that $f$ and $E_4^3$ are linearly independent for all fundamental discriminant $D > 0$. This finishes the proof of Theorem \ref{thm:main} for $d = 2$.
\end{proof}

Using the calculation in section \ref{sec:Pet} and the correspondence in section \ref{sec:3}, we can prove the following lemma.
\begin{lemma}
  \label{lemma:estimate}
  For $d = 3, k \ge 3$ and $f(\tau) = \sum_{n \ge 1} c_f(n) q^n \in S_{3 k}$, let $c_f > 0$ be a constant such that  
  $$|c_f(n)| \le c_f \cdot n^{3k/2}$$
  for all $n \ge 1$.
  Then we have the bound
  \begin{equation}
    \label{eq:estimate}
  |  \langle E^\Delta_{F, k}, f\rangle|
    \le C_k c_f D_F^{-k/4}
  \end{equation}
  for all cubic field $F$, with $C_k := 6c_k \frac{\zeta(k/2)^3}{\zeta(k)^2} \zeta(3k/2 - 1)$ and the constant $c_k$ given in \eqref{eq:ck}.
\end{lemma}

\begin{proof}
  Let $a_k:=\tfrac{\Gamma(3k-1)}{\Gamma(k)}(4\pi)^{2-3k}$.
For $\beta \in S_0 \subset F$, recall that $f_\beta$ is the binary cubic form associated to it in \eqref{eq:fs}, which has coefficients $A_\beta, B_\beta, C_\beta, D_\beta$.
  Using \eqref{eq:inner0}, the estimate in Example \ref{exmp:d=3} and \eqref{eq:Pbd}, we obtain the bound
  \begin{align*}
    |    \langle E^\Delta_{F, k}, f\rangle|
    & \le
    a_k\sum_{\ell = 0}^{k-1}(2\pi)^{k-1-\ell} 
    \sum_{n \ge 1}
      \frac{    |c_f(n)|}{    n^{2k + \ell }}
      \sum_{\beta \in S_0}
      A_\beta^{-k}
      \left|\sum_{j' = 1}^{d}      
\frac{P_{d-1, k, \ell}}{Q_{d - 1, k+\ell}}(\beta_{j'} - \beta_1, \dots, \beta_{j'} - \beta_{d})
    \ebf(n\beta_{j'})\right|\\
    &\le
c_f    \cdot a_k\sum_{\ell = 0}^{k-1}(2\pi)^{k-1-\ell} \zeta(k/2 + \ell)
      \frac{      \binom{k-1+\ell}{\ell}}{\binom{k+2\ell}{\ell}}
      (\ell + 1)!
\frac{3^{k/2 + \ell}}{2}\\
&\times      \sum_{\beta \in S_0}
    A_\beta^{-k}
                                                   \frac{((\beta_{  1} - \beta_{  2})^2 + (\beta_{  2} - \beta_{  3})^2 + (\beta_{  3} - \beta_{  1})^2)^{k/2 + \ell}}{((\beta_{  1} - \beta_{  2})^2 (\beta_{  2} - \beta_{  3})^2 (\beta_{  3} - \beta_{  1})^2)^{(k + \ell)/2}}\\
    & \le
2^{-1}    c_f\cdot a_k\sum_{\ell = 0}^{k-1}(2\pi)^{k-1-\ell}\zeta(k/2+\ell)
     \frac{      \binom{k-1+\ell}{\ell}}{\binom{k+2\ell}{\ell}}
      (\ell + 1)!
6^{k/2 + \ell}
      \sum_{\beta \in S_0}
      \frac{P(f_\beta)^{k/2 + \ell}}{
\Delta(f_\beta)^{(k + \ell)/2}}\\    
    & \le
      c_f\cdot c_k\sum_{\beta \in S_0}
      \Delta(f_\beta)^{-k/4}
    \le c_f \cdot c_k \cdot 2 |\Aut(\Oc_F)|\cdot D_F^{-k/4}\eta_F\lp \tfrac{k}{2}\rp
  \end{align*}
  Here the constant $c_k$ is defined by
  \begin{equation}
    \label{eq:ck}
    c_k :=  \frac{\Gamma(3k-1)}{2 \Gamma(k)}(4\pi)^{2-3k}
    \sum_{\ell = 0}^{k-1}(2\pi)^{k-1-\ell}\zeta(k/2+\ell)
     \frac{      \binom{k-1+\ell}{\ell}}{\binom{k+2\ell}{\ell}}
      (\ell + 1)!
6^{k/2 + \ell}.
  \end{equation}
  For the last steps, we used Prop.\ \ref{prop:DF}.
  Combining this with \eqref{eq:DW} and applying $\zeta_F(s) \le \zeta(s)^3$ for $s > 1$, we have
  \begin{align*}
    |    \langle E^\Delta_{F, k}, f\rangle|
    &\le
      c_f c_k  2 |\Aut(\Oc_F)|     \frac{\zeta_F\bigl(\tfrac{k}{2}\bigr)}{\zeta_F(k)} \zeta(k)\zeta\bigl(\tfrac{3k}{2}-1\bigr)
      D_F^{-k/4}
\le 6 c_f c_k\frac{\zeta\bigl(\tfrac{k}{2}\bigr)^3}{\zeta(k)^2}\zeta\bigl(\tfrac{3k}{2}-1\bigr)
D_F^{-k/4}
  \end{align*}
for $k \ge 3$. This finishes the proof.
\end{proof}

\begin{rmk}
  \label{rmk:k=4}
  For $k = 4$, the bound above gives $C_4 < 5.79 $.
  We can obtain a better bound by estimating the second to the last line in Example \ref{exmp:d=3} case by case  for each $\ell = 0, 1, 2, 3$, instead of using \eqref{eq:ineq1}.
  The improved bound is
  $$
  |  \langle E^\Delta_{F, 4}, f\rangle|
  \le 0.067 c_f D_F^{-1}
  $$
  for all totally real cubic field $F$.
\end{rmk}
Now we are ready to prove Theorem \ref{thm:main} in the cubic case.

\begin{proof}[Proof of Theorem \ref{thm:main} for $d = 3$]
  Since~$\Mc_\Qb^{(12)}=\Cb[E_{12},\Delta]$, we only have to check that $(\Mc^\theta_{F})^\Delta \cap \Mc_\Qb^{(12)}$ is 2 dimensional. 
  For any $L \in \Uc^+_F$, the diagonal restriction $\theta^\Delta_L$ is the theta function for a unimodular lattice $P$ over $\Zb$.
  So we know that $\theta_P \in (\Mc^\theta_{F})^\Delta$ for some Niemeier lattice $P$.
  To see that it is linearly independent from $E_{F, 4}^\Delta = 1 + c(1) q + O(q^2)$, it suffices to show that $c(1)$ is not integral.
  We have checked this numerically for any cubic $F$ with $D_F < 70000$.

More generally, we have
$$
\theta_P = E_{12} + (N_{2}(P) - 65520/691) \Delta,
$$
with $N_{2}(P)$ is the number of norm 2 vectors in $P$.
From Table V in \cite{CS82}, we obtain a list of $N_2(P)$ and 
$$
|\langle \theta_P, \Delta \rangle| = |N_{2}(P) - 65520/691| \langle \Delta, \Delta \rangle > 1.22 \times 10^{-6}
$$
for any Niemeier lattice $P$.
On the other hand by taking $c_\Delta = 1$, the upper bound found in Lemma~\ref{lemma:estimate} and improved in Remark \ref{rmk:k=4} gives us
\[
|\langle E^\Delta_{F, 4}, \Delta\rangle|\;<\;\frac{0.067}{D_F}.
\]
So $E^\Delta_{F, 4}$ and $\theta_P$ are linearly independent for $D_F \ge 60000$.
This finishes the proof.
\end{proof}

\section{Numerical Evidence for Conjecture \ref{conj:d}}
In this section, we approach numerically Conjecture \ref{conj:d} in the case~$F$ is a totally real field of degree $d \in \{4, 5, 6\}$. 
For these choices of~$d$ the space~$\mathcal{M}_\Qb^{(4d/d_2)}$ can be in principle generated by the restriction of Eisenstein series and of (at most) one theta function~$\theta_{L}$ of rank~$8/d_2$.  Conjecture~\ref{conj:d} reduces then to the verification of the linear independence of~$\theta_{L}^\Delta$ and~$E_{F,4/d_2}^\Delta$ for~$d = 5, 6$, and of monomials in~$\theta_{L}^\Delta, E_{F,4d/d_2}^\Delta$ and~$E_{F,k}^\Delta$ in general for suitable weights~$k$. This approach gives data supporting Conjecture~\ref{conj:d} in the case~$d=4,5$, and in the case~$d=6$ except for two fields $F$.
Our result, for which evidence is given in this final section, is the following.
\begin{thm}
\label{thm:num}
Conjecture~\ref{conj:d} holds for
\begin{enumerate}
\item $d=4$ and~$D_F\le 10^5$;
\item $d=5$ and~$D_F\le 2\times10^6$;
\item $d=6$ and~$D_F\le 5\times10^6$ except for the fields of discriminant~$453789$ and~$1397493$.
\end{enumerate}
\end{thm}

\subsection{A note on the computations}
For $l,k\in\Zb_{\ge0}$, let $\sigma_{k-1}$ be as in Remark \ref{rmk:FEHE} and define
\[
s_l^F(k)\;:=\;\sum_{\substack{\nu\in\mathfrak{d}_F^{-1}\\ \nu\gg 0 \\ \mathrm{tr}(\nu)=l}}\sigma_{k-1}\bigl((\nu)\mathfrak{d}_F\bigr)\,.
\]
Then the diagonal restriction of~$E_{F,k}$ has the following~$q$-expansion at~$\infty$ by \eqref{eq:FEHE} 
\begin{equation}
\label{eqn:expres}
E_{F,k}^\Delta(\tau)\;=\;1\;+\;\frac{2^d}{\zeta_K(1-k)}\sum_{l=0}^\infty{s^F_l(k)}\,.
\end{equation}
We computed the first few coefficients of the above expansion with PARI/GP \cite{PARI}.
As~\eqref{eqn:expres} shows, this reduces to the determination of the functions~$s^F_l(k)$ for small values of~$l$ (up to~$l=5$ in the case $d=5$) and different values of~$k$. The main difficulty is to find the totally positive~$\nu\in\mathfrak{d}_F^{-1}$ of fixed trace~$l$. 
Let~$(\nu_1,\dots,\nu_d)$ be an integral basis for~$\mathfrak{d}_F^{-1}$. Then any~$\nu\in\mathfrak{d}_F^{-1}$ is of the form~$\nu=v_1\nu_1+\cdots+v_d\nu_d$ for~$(v_1,\dots,v_d)\in\Zb^d$ and conversely every vector in~$\Zb^d$ gives an element~$\nu\in\mathfrak{d}_F^{-1}$. If~$Q(x_1,\dots,x_d)$ denotes the quadratic form~$x_1^2+\dots x_d^2$, we have, for a totally positive~$\nu\in\mathfrak{d}_F^{-1}$, that~$Q(\sigma_1(\nu),\dots,\sigma_d(\nu))<\tr(\nu)^2$. This implies that if~$A=(\sigma_i(\nu_j))_{i,j}$ denotes the matrix of the real embeddings of the basis of~$\mathfrak{d}_F^{-1}$, we can search the totally positive~$\nu\in\mathfrak{d}_F^{-1}$ of fixed trace~$l$ among of vectors~$v=(v_1,\dots,v_d)\in\Zb^d$ satisfying
\[
v^T(A^TA)v\;=\;Q(\nu)\;<\;l^2\,.
\]
This gives a finite (but large as~$l$ and~$D_F$ grow) set of vectors on which we can perform the final search. Once the suitable~$\nu\in\mathfrak{d}_F^{-1}$ have been determined, it is straightforward to compute~$\sigma_k\bigl((\nu)\mathfrak{d}_F\bigr)$ for every value of~$k$ by using the basic PARI functions.


\begin{rmk}
\label{rmk:810}
It is possible to investigate also the cases~$d=8,10$ with the method outlined at the beginning of this section.
For the case~$d=8$ we need to compute five coefficients of the~$q$-expansion~\eqref{eqn:expres}, while for~$d=10$ we need to compute six coefficients. This, together with the size of the discriminants of these fields ($D_F\ge 282300416$ for~$d=8$ and~$D_F\ge443952558373$ for~$d=10$), makes it hard to collect significant data in these cases. 
\end{rmk}

\subsection{Tables}
\subsubsection*{d=4} Let~$F$ be a totally real field with~$[F:\Qb]=4$. 
In this case, the proof of Conjecture~\ref{conj:d} reduces to the statement that~$\mathcal{M}_\Qb^{(8)}$ is spanned by restrictions of Hilbert Eisenstein series on~$\Gamma_F$. It is easy to see that~$\{E_4^2,\Delta E_4,\Delta^2\}$is a generating set for~$\mathcal{M}_\Qb^{(8)}$. By a dimension argument, $E_{F,2}^\Delta=E_4^2$. It follows that~$\Delta E_4$ and~$\Delta^2$ can be obtained by restriction of Eisenstein series on~$\Gamma_F$ respectively if the sets $\{E_{F,4}^\Delta, (E_{F,2}^\Delta)^2\}$, and~$\{(E_{F,2}^\Delta)^3,E_{F,2}^\Delta E_{F,4}^\Delta, E_{F,6}^\Delta\}$ are both linearly independent. 

In order to study this problem, we compute the restriction of~$E_{F,k}$ for~$k=4,6$. 
As bases for~$M_{16}$ and $M_{24}$, we choose~$\{E_4^4,E_4\Delta\}$ and $\{E_4^6,E_4^3\Delta,\Delta^2\}$ respectively. We have
\begin{equation}
\label{eqn:resd=4}
\begin{aligned}
E_{F,4}^\Delta&\;=\; E_4^4\;+\;b E_4\Delta\,,\\
E_{F,6}^\Delta&\;=\; E_4^6\;-\;c_1 E_4^3\Delta\;+\;c_2 \Delta^2\,,
\end{aligned}
\end{equation}
for some coefficients~$b,c_1,c_2\in\Qb$ that depend on $F$.
To prove Conjecture \ref{conj:d}, it suffices to check that $b$ and $c_2$ are both non-zero. We computed the coefficients~$b,c_1,c_2$ for the first~$30$ totally real quartic fields~$F$. The results are reported in Table~\ref{table:d=4}. For these fields it is enough to specify the discriminant~$D_F$ to uniquely identify the field~$F$ (check the number field database~\cite{lmfdb}). This remark applies also for the fields we consider in the cases~$d=5,6$.

It turns out that the numerical values of~$b,c_1$, and $c_2$ are very close to~$955,1439$, and~$-129930$ respectively. These numbers are related to the Eisenstein series of weight~$16$ and~$24$ since
\[
E_{16}\;=\;E_4^4+b(E_{16})E_4\Delta\,,\quad E_{24}\;=\;E_4^5+c_1(E_{24})E_4^2\Delta+c_2(E_{24})\Delta^2,
\]
with
\[
b(E_{16})=-\tfrac{3456000}{3617}\sim 955\,,\;c_1(E_{24})=\tfrac{340364160000}{236364091}\sim1439\,,\;c_2(E_{24})=-\tfrac{30710845440000}{236364091}\sim129930\,.
\]
In other words, it seems that the diagonal restriction of~$E_{F,4}$ and~$E_{F,6}$ are close to~$E_{16}$ and~$E_{24}$ respectively. 
In analogy with the proof of Theorem~\ref{thm:main} in the case~$d=3$, Conjecture~\ref{conj:d} holds for~$D_F\gg0$ if the Petersson products of~$E_{F,4}^\Delta$ and~$E_{F,6}^\Delta$ with all cusp forms of weight~$16$ and~$24$ respectively can be bounded by small quantities as~$D_F\to\infty$. 
If~$F$ ranges over the totally real quartic fields with no non-trivial subfields, the decay of the Petersson products as~$D_F\to\infty$ can be observed from the data.
We expect similar strategy for the proof of Theorem \ref{thm:main} when $d = 3$ to work in this case.
When~$F$ ranges instead over extensions of the form~$\Qb\subset K \subset F$, where~$K$ is a fixed real quadratic field, the data suggest that
\[
\langle E_{F,k}^\Delta,f \rangle\rightarrow \langle E_{K,2k}^\Delta,f \rangle\quad\text{as }\mathrm{disc}(F)\to\infty\,.
\] 
The proof of Conjecture~\ref{conj:d} may be obtained then in two steps: first proving that~$E_{F,k}$ restrict to the Hilbert Eisenstein series~$E_{K,2k}$ on~$\Gamma_K$ as~$F\to\infty$, and then using Theorem~\ref{thm:main} for the real quadratic field~$K$.

\begin{table}[h]
\caption{$d=4$}
\label{table:d=4}
$\begin{array}{lcccc}
\toprule
D_F & \multicolumn{2}{c}{E_{F,4}^\Delta} &\multicolumn{2}{c}{E_{F,6}^\Delta}\\[1ex]
& -b & |b-b(E_{16})| & |c_1-c_1(E_{24})| & |c_2-c_2(E_{24})|\\
\midrule
725&\frac{518400}{541}&2.7375349&0.00050313732&25.886498\\[1ex]
1125&\frac{1209600}{1261}&3.7507260&0.00054525118&81.739221\\[1ex]
1600&\frac{16588800}{17347}&0.80418080&0.00021600333&72.207992\\[1ex]
1957&\frac{3379968}{3541}&0.96439255&0.00038594892&17.453573\\[1ex]
2000&\frac{3628800}{3793}&1.2217550&0.00025214822&55.822134\\[1ex]
2048&\frac{83358720}{87439}&2.1522766&0.00086000436&17.157301\\[1ex]
2225&\frac{4406400}{4601}&2.2168733&0.00044417599&65.944997\\[1ex]
2304&\frac{6996480}{7337}&1.8993132&0.00078107824&34.635539\\[1ex]
2525&\frac{40953600}{42787}&1.6625629&0.00038679430&60.388956\\[1ex]
2624&\frac{31242240}{32681}&0.48766988&0.00016760431&11.280096\\[1ex]
2777&\frac{30326400}{31739}&0.0052682944&2.49791\times10^{-5}&3.1916173\\[1ex]
3600&\frac{3940800}{4117}&1.7138725&0.00032163274&63.391164\\[1ex]
3981&\frac{22598400}{23651}&0.0065088042&1.13683\times10^{-5}&16.924484\\[1ex]
4205&\frac{81112320}{84937}&0.51758364&7.99169\times10^{-5}&20.235531\\[1ex]
4225&\frac{31168800}{32567}&1.5789962&0.00036468807&64.303710\\[1ex]
4352&\frac{14613696}{15301}&0.40686765&0.00042977636&2.3880196\\[1ex]
4400&\frac{287193600}{300017}&1.7697819&0.00034047188&63.576584\\[1ex]
4525&\frac{315705600}{329717}&2.0167958&0.00041971303&62.764065\\[1ex]
4752&\frac{94772160}{99107}&0.77303737&0.00019694052&7.4011277\\[1ex]
4913&\frac{358572096}{375437}&0.40870317&4.15983\times10^{-5}&2.8025931\\[1ex]
5125&\frac{24364800}{25453}&1.7587165&0.000037508012&63.196773\\[1ex]
5225&\frac{262310400}{273971}&1.9505876&0.00039428701&63.490510\\[1ex]
5725&\frac{716947200}{748883}&1.8674479&0.00042362838&63.447454\\[1ex]
5744&\frac{727626240}{761737}&0.26820601&7.42018\times10^{-5}&5.6160966\\[1ex]
6125&\frac{454636800}{474913}&1.8174699&0.00042240976&63.162128\\[1ex]
6224&\frac{204809472}{214357}&0.028287048&2.32205\times10^{-5}&5.7256495\\[1ex]
6809&\frac{87570720}{91723}&0.75775312&0.00019944285&7.0978686\\[1ex]
7053&\frac{1504154880}{1573751}&0.28894424&0.00013645348&2.0848135\\[1ex]
7056&\frac{191034720}{200123}&0.90144417&0.00037862134&11.709300\\[1ex]
7168&\frac{670104576}{701855}&0.72584168&0.00033000104&3.6107465\\[1ex]
\bottomrule
\end{array}$
\end{table}

\subsubsection*{d=5} Let~$F$ be a totally real field of degree~$5$. The space~$\mathcal{M}_\Qb^{(20)}$ is generated by the set~$\{E_{20}, E_8\Delta, E_4\Delta^3,\Delta^5\}$. In order to get this space by restriction of Hilbert theta series (Conjecture~\ref{conj:d}), we only need to consider a Hilbert theta function~$\theta_L$ for $L \in \Uc^{+, 8}_F$
and the Eisenstein series~$E_{F,4},E_{F,8},$ and~$E_{F,12}$.
Fixing basis for~$M_{20},M_{40},$ and~$M_{60}$, we find the expressions
\begin{equation}
\begin{aligned}
E_{F,4}^\Delta&\;=\;E_4^5\;+\;b E_4^2\Delta\,,\\
E_{F,8}^\Delta&\;=\;E_4^{10}\;+\;c_1E_4^7\Delta\;+\;c_2E_4^4\Delta^2\;+\;c_3E_4\Delta^3\,,\\
E_{F,12}^\Delta&\;=\;E_4^{15}\;+\;d_1E_4^{12}\Delta\;+\;d_2E_4^9\Delta^2\;+\;d_3E_4^6\Delta^3\;+\;d_4E_4^3\Delta^4\;+\;d_5\Delta^5\,,
\end{aligned}
\end{equation}
for~$b,c_i,d_i\in\Qb$ that depends on~$F$. Since~$\theta_L^\Delta=1+\sum_{n\ge1}{a_nq^n}$ with~$a_n\in\Zb$, in order to prove linear independence of~$\theta_L^\Delta$ and~$E_{F,4}^\Delta$, it suffices to show that~$b\not\in\Zb$. If this holds true, we only need that~$c_3\neq0$ and~$d_5\neq 0$ to prove Conjecture~\ref{conj:d}. The results of the computation of~$b,c_3,$ and~$d_5$, for the first~30 totally real quintic fields~$F$ (ordered by discriminant) can be found in Table~\eqref{table:d=5}.
Similarly to the case~$d=4$, the numerical values of~$b,c_i,d_i$ are close to the coefficients appearing in the expression of the Eisenstein series~$E_{20},E_{40}$, and~$E_{60}$ with respect to the bases specified above:
\[
E_{20}=E_4^5+b(E_{20})E_4^2\Delta\,,\;E_{40}=E_4^{10}+\sum_{i=1}^3{c_i(E_{40})E_4^{10-3i}\Delta^i}\,,\;E_{60}=E_4^{15}+\sum_{i=1}^5{d_i(E_{60})E_4^{15-3i}\Delta^i}\,,
\]
the relevant values being
\[
\begin{aligned}
b(E_{20})&=\tfrac{209520000}{174611}\sim 1199\,,\quad c_3(E_{40})=\tfrac{27014542428753690624000000000}{261082718496449122051}\sim103471200\\ 
d_5(E_{60})&=\tfrac{1423152253904739393602157818174020937318400000000000000}{1215233140483755572040304994079820246041491}\sim1171094011917\,.
\end{aligned}
\]
In Table~\ref{table:d=5} we do not write the numerical values of~$c_3,d_5$, but of their difference with the coefficients~$c_3(E_{40})$ and~$d_5(E_{60})$ respectively.
Analogously to the case~$d=4$, it seems that the diagonal restriction of~$E_{F,4},E_{F,8}$, and~$E_{F,12}$ are close to the Eisenstein series~$E_{20},E_{40}$, and~$E_{60}$ respectively. 
In particular, since~$\langle E_{20},E_4^2\Delta\rangle=0$, this implies that the Petersson product
\[
\langle E_{F,4}^\Delta,E_4^2\Delta\rangle\;=\bigl|b-b(E_{20})\bigr|\langle E_4\Delta^2,E_4^2\Delta\rangle\
\]
is small for any field~$F$ and may decay as~$D_F\to\infty$. Similar considerations apply to the cases~$E_{F,8}^\Delta$ and~$E_{F,12}^\Delta$. 
\begin{table}
\caption{$d=5$}
\label{table:d=5}
$\begin{array}{ccccc}
\toprule
D_F & \multicolumn{2}{c}{E_{F,4}^\Delta} & E_{F,8}^\Delta & E_{F,12}^\Delta \\
 & -b & |b-b(E_{20})| & |c_3-c_3(E_{40})| & |d_5-d_5(E_{60})| \\
\midrule
14641&\frac{1017360000}{847811}&0.060027104&20.049846&602.44929\\[1ex]
24217&\frac{539084160}{449263}&0.0056153314&3.0959986&626.69793\\[1ex]
36497&\frac{228998016}{190847}&0.020731861&3.7691249&625.79357\\[1ex]
38569&\frac{1372671360}{1144027}&0.065169297&7.1399961&53.593811\\[1ex]
65657&\frac{17909631360}{14926259}&0.050318395&1.4956754&21.447253\\[1ex]
70601&\frac{22786945920}{18989939}&0.023943997&6.3509437&57.580627\\[1ex]
81509&\frac{1255163040}{1046047}&0.013653680&0.14871660&33.681371\\[1ex]
81589&\frac{157427145}{131198}&0.0040921029&3.7633773&5.5844793\\[1ex]
89417&\frac{3299933520}{2750093}&0.010842438&0.88686411&7.1415794\\[1ex]
101833&\frac{27422375040}{22853437}&0.00095029875&2.8922290&147.56474\\[1ex]
106069&\frac{8416776960}{7014301}&0.020817098&1.2397693&6.2320866\\[1ex]
117688&\frac{72647616960}{60544963}&0.029096490&0.024328539&11.486101\\[1ex]
122821&\frac{2646596160}{2205599}&0.019992669&3.3965204&46.863235\\[1ex]
124817&\frac{169474446720}{141236923}&0.0057786083&1.2672930&6.7827247\\[1ex]
126032&\frac{186909793920}{155769041}&0.0082151643&0.36230593&11.688370\\[1ex]
135076&\frac{39368816640}{32809823}&0.014954319&0.35027520&21.034030\\[1ex]
138136&\frac{42439256640}{35368523}&0.0084035031&0.15151378&21.202474\\[1ex]
138917&\frac{30923687520}{25771127}&0.010994082&3.3294478&176.27092\\[1ex]
144209&\frac{35105335200}{29256611}&0.013203831&1.2246131&5.0019384\\[1ex]
147109&\frac{79422612480}{66189911}&0.0041847312&1.4237043&8.3262747\\[1ex]
149169&\frac{316249522560}{263551583}&0.028634117&1.8194260&15.033097\\[1ex]
153424&\frac{24509153664}{20425187}&0.023174277&3.4016473&25.402792\\[1ex]
157457&\frac{76544072064}{63790577}&0.0031668548&1.0412335&7.2747406\\[1ex]
160801&\frac{411236196480}{342716341}&0.0072814568&0.29819082&10.519558\\[1ex]
161121&\frac{6653973120}{5545309}&0.0038819428&0.12885785&18.108623\\[1ex]
170701&\frac{125695281600}{104754347}&0.019242412&2.2046583&104.80927\\[1ex]
173513&\frac{530059904640}{441734773}&0.026193754&0.10356861&42.247614\\[1ex]
176281&\frac{187387136640}{156166489}&0.0054077065&1.9881737&26.798380\\[1ex]
176684&\frac{60248727936}{50210921}&0.011580917&1.0387400&12.861239\\[1ex]
179024&\frac{638510843520}{532132229}&0.014289468&0.34114599&1.2188617\\[1ex]
\bottomrule
\end{array}$
\end{table}

\subsubsection*{d=6} We have that~$\mathcal{M}_\Qb^{12}=\Cb[E_4^3,\Delta]$. We have only to check that 
\[
E_{F,2}^\Delta\;=\;E_4^3\;+\;b\cdot\Delta
\]
is not the restriction of a Hilbert theta function~$\theta_L$. We know this is the case if~$b$ is not an integer, as explained in the proof of Theorem~\ref{thm:main} in the case~$d=3$. However, looking at the values of~$b$ computed for the first~$30$ totally real sextic fields~$F$ in Table~\eqref{table:d=6}, this is not always the case. Since~$\theta_L^\Delta=1+N_2(L)q+\cdots$, we have to compare, for integral values of~$b$, the number~$720+b$ with the possible values of~$N_2(L)$ listed in table V of~\cite{CS82} to check whether they differ or not. This happens in all cases but two: the field of discriminant~$453789$ has~$720+b=0=N_2(\Lambda_{24})$, the field of discriminant~$1397493$ has~$720+b=72=N_2(A^{12}_2)$. For these fields our argument can not confirm the validity of Conjecture~\ref{conj:d}. We checked fields up to $D_F=5\times10^6$ (144 fields) and found no other such instances.

As in the cases~$d=4,5$, in table~\eqref{table:d=6} we also compare the value of~$b$ with~$b(E_{12})=-\tfrac{432000}{691}$ (see~\eqref{eq:E43}).

\begin{table}[h]
\caption{$d=6$}
\label{table:d=6}
\[
\begin{array}{ccc}
\toprule
D_F & \multicolumn{2}{c}{E_{F,2}^\Delta}\\
 & -b & |b-b(E_{12})|\\
\midrule
300125&\frac{21600}{37}&41.397113\\[1ex]
371293&\frac{11808}{19}&3.7072130\\[1ex]
434581&\frac{8352}{13}&17.280641\\[1ex]
453789&720&94.819103\\[1ex]
485125&\frac{7200}{11}&29.364557\\[1ex]
592661&672&46.819103\\[1ex]
703493&\frac{2048}{3}&57.485769\\[1ex]
722000&\frac{4800}{7}&60.533388\\[1ex]
810448&\frac{3456}{5}&66.019103\\[1ex]
820125&\frac{43200}{73}&33.400075\\[1ex]
905177&\frac{3348}{5}&44.419103\\[1ex]
966125&675&49.819103\\[1ex]
980125&675&49.819103\\[1ex]
1075648&\frac{8352}{13}&17.280641\\[1ex]
1081856&\frac{3072}{5}&10.780897\\[1ex]
\bottomrule
\end{array}\quad
\begin{array}{ccc}
\toprule
D_F & \multicolumn{2}{c}{E_{F,2}^\Delta}\\
 & -b& |b-b(E_{12})|\\
\midrule
1134389&684&58.819103\\[1ex]
1202933&608&17.180897\\[1ex]
1229312&\frac{27264}{43}&8.8656144\\[1ex]
1241125&\frac{28800}{47}&12.414940\\[1ex]
1259712&\frac{17280}{31}&67.761542\\[1ex]
1279733&\frac{11736}{17}&65.172044\\[1ex]
1292517&\frac{16416}{29}&59.111932\\[1ex]
1312625&\frac{9000}{13}&67.126795\\[1ex]
1387029&696&70.819103\\[1ex]
1397493&648&22.819103\\[1ex]
1416125&\frac{12000}{19}&6.3980501\\[1ex]
1528713&\frac{12096}{19}&11.450682\\[1ex]
1541581&\frac{8352}{13}&17.280641\\[1ex]
1683101&\frac{65088}{103}&6.7414328\\[1ex]
1767625&\frac{25200}{41}&10.546751\\[1ex]
\bottomrule
\end{array}
\]
\end{table}

\bibliography{span}{}
\bibliographystyle{alpha}
\end{document}